\documentclass[smallextended,12pt]{svjour3}       
\smartqed  
\usepackage{graphicx}
\usepackage{epsfig}
\usepackage[left=2.5cm,right=2.5cm]{geometry}
\usepackage{amssymb}
\usepackage{mathrsfs}
\usepackage{indentfirst}
\usepackage{graphicx}
\usepackage{amsmath}
\usepackage{color}
\usepackage{booktabs}
\usepackage{marvosym}
\usepackage{enumitem}
\usepackage{subfigure}
\usepackage{thmtools}
\usepackage{subeqnarray,setspace}
\usepackage{epstopdf}
\journalname{}
\begin{document}

\title{A robust optimal control problem with moment constraints on distribution:
theoretical analysis and an algorithm
}


\author{Jianxiong Ye \and
      Lei Wang \and Changzhi Wu\\ \and Jie Sun \and Kok Lay Teo \and  Xiangyu Wang 
}


\institute{Jianxiong Ye \at
              School of Mathematics and Computer Sciences, Fujian Normal University,
  Fuzhou, People's Republic of China\\
School of Energy and Power Engineering, Huazhong University of Science and Technology, Wuhan, People's Republic of China \\
               \email{jxye@fjnu.edu.cn}             \\
Lei Wang({\large\Letter}) \at
              School of Mathematical Sciences, Dalian University of Technology,
  Dalian, People's Republic of China \\
              \email{wanglei@dlut.edu.cn}             \\
             Changzhi Wu, Xiaoyu Wang\at
              Australasian Joint Research Centre for Building Information Modelling, School of Built Environment, Curtin University, Australia\\
\email{C.Wu@exchange.curtin.edu.au}\\
Jie Sun, Kok Lay Teo\at
              Department of Mathematics and Statistics, Curtin University, Perth, Australia\\
              \email{jie.sun@curtin.edu.au}\\
Kok Lay Teo\at
              \email{K.L.Teo@curtin.edu.au}\\
              Xiaoyu Wang\at
            \email{Xiangyu.Wang@curtin.edu.au}\\
}

\date{Received: date / Accepted: date}

%

\maketitle

\begin{abstract}
We study an optimal control problem in which both the objective
function and the dynamic constraint contain an uncertain
parameter. Since the distribution of this uncertain parameter  is
not exactly known, the objective function is taken as the
worst-case expectation over a set of possible distributions of the
uncertain parameter. This  ambiguity set of distributions is, in
turn, defined by the first two moments of the random variables
involved. The optimal control is found by minimizing the
worst-case expectation over all possible distributions in this
set. If the distributions are discrete, the stochastic min-max
optimal control problem can be  converted into a convensional
optimal control problem via duality, which is then approximated as
a finite-dimensional optimization problem via the control
parametrization. We derive necessary conditions of optimality and
propose an algorithm to solve the approximation optimization
problem. The results of discrete probability distribution are then
extended to the case with one dimensional continuous stochastic
variable by applying the control parametrization methodology on
the continuous stochastic variable, and the convergence results
are derived.
 A numerical example
is present to illustrate the potential application of the proposed
model and the effectiveness  of the algorithm.

\end{abstract}

\section{Introduction}
\label{intro}\noindent Ideas  to immunize optimization problems
against perturbations in model parameters arose as early as in
1970s. A worst-case model for linear optimization such that
constraints are satisfied under all possible perturbations of the
model parameters was proposed in \cite{Soyster}. A common approach
to solving this type of models is to transform the original
uncertain optimization problem into a deterministic convex
program. As a result, each feasible solution of the new program is
feasible for all allowable realizations of the model parameters,
therefore the corresponding solution tends to be rather
conservative and in many cases even infeasible. For a detailed
survey, see the recent monograph \cite{bental}.

For traditional stochastic programming approaches, uncertainties
are modeled as random variables with known distributions. In very
few cases,  analytic solutions are obtained (see, e.g. Birge and
Louveaux \cite{ISP}, Ruszczynski and Shapiro \cite{SP}). These
approaches may not be always applicable in practice, as the exact
distributions of the random variables are usually unknown.

In the framework of robust optimization,  uncertainties are
usually modeled by uncertainty sets, which specify certain ranges
for the random variables. The worst-case approach is used to
handle the uncertainty. It is often computationally advantageous
to use the ``robust" formulation of the problem. However, the use
of  uncertainty sets as the possible supporting sets for the
random variables is restrictive in practice; it leads to
relatively conservative solutions.

The recently developed ``distributionally robust" optimization
approach combines the philosophies of traditional stochastic and
robust optimization -- this approach does not assume uncertainty
sets, but keep using the worst-case methodology. Instead of
requiring the shape and size of the support sets for the random
variables, it assumes that the distributions of the random
variables satisfying a set of constraints, often defined in terms
of  moments and supports. Since the first two moments can usually
be  estimated via statistical tools, the distributionally robust
model appears to be more applicable in practice. Furthermore,
since it takes the worst-case expected cost, it inherits
computational advantages from robust optimization. Due to these
advantages, distributionally robust optimization has attracted
more and more attention in operations research community
\cite{Goh14,Sim,Sim07,Sim09,El,Ye,Zymler,Mehrotra}.

 In this
paper, we propose a novel optimal control model with an uncertain
parameter for which its exact distribution is unknown. However, it
is assumed that  the mean and the standard deviation of the
uncertain parameter are known. The optimal control is found by
minimizing the worst-case expectation with respect to all
distributions in an ``ambiguity set". Both the problems with
discrete probability distribution and with continuous probability
distribution will be discussed. We first consider the case of
discrete probability distribution, in which the min-max optimal
control problem  is transformed  into an equivalent finite
dimensional minimization problem via duality. Then the necessary
conditions of optimality are derived. The results for the case of
discrete probability distribution are then extended to the case
with one dimensional continuous stochastic variable. The control
parametrization methodology is applied to parameterise the
continuous stochastic variable. Finally,  an example is solved
showing the potential application of the proposed optimal control
framework and the effectiveness of the algorithm.

\section{{\color{blue}Problem statement}}
\label{sec:1}
 For simplicity, we only discuss optimal control of
 dynamical systems with a single uncertain parameter.
 However, the results  can be directly extended to cases
 involving multiple independent uncertain parameters.

To begin with, consider a system of  ordinary differential
equations with an uncertain  parameter as follows:
\begin{align}\label{s1} \left\{\begin{array}{ccc}
\dot{x}(t)=&f(x,u,p)\\
x(0)=&x^0~~~~~~~~
\end{array}
\right.
\end{align}
where $x\in R^{n_x}$ is the state vector, $u\in R^{n_u}$ is the
control vector function, and $p\in R$ is an uncertain parameter.
In general, the parameter $p$ is regarded as uniquely determined.
In reality, however, this hypothesis often does not hold, since
parameter $p$ is  uncertain subject to  variability. The only
reliable information is that the value of the parameter falls
within a certain range and that its potential values follow some
statistical distribution. Our interest focuses on the following
distributionally robust optimal control problem.
\begin{eqnarray}
(\mbox{DROCP}):&&\inf_{u} \sup_{F}\,\,\,J(u,F)\triangleq\mathbb{E}_{F} h(x(t_f;u,p)) \nonumber \\
&& {\rm s.t.}\,\,\,\dot{x}(t)=f(x,u,p), \ \ \ \ t\in [0,t_f],\ \ x(0)=x^0,\label{DR-SC}
\\&&\ \ \ \ \ \ p\sim F\in \mathcal{F}(\mu,\sigma^2)=\{F:\mathbb{E}_{F}(p)=\mu,
\mathbb{E}_{F}(p-\mu)^2=\sigma^2\} ,\label{DR-FC}
\\ &&\ \ \ \ \ \ u(t)\in U\subset R^{n_u}.\label{DR-UC}
\end{eqnarray}
Here, $U$  is a compact and convex subset of $R^{n_u}$. The
difference between Problem (DROCP) and the standard optimal
control problem is that the parameter  $p$ herein is considered as
a stochastic variable  with  distribution $F$. The distribution
 $F$, however, is not exactly known. The only knowledge, which is available, is the
mean  and the standard deviation of the  distribution $F$; they
are denoted by $\mu$   and $\sigma$, respectively. The set of all
such distributions  is denoted by $\mathcal{F}(\mu,\sigma)$. Any
measurable function defined in $[0,t_f]$ with values in $U$ is
called an admissible control. Let $\mathcal{U}$ be the class of
all such admissible controls.

\begin{remark} It is sufficient to discuss  the objective function in Mayer form,
because the problems in  Bolza or Lagrange form
 can  be transformed into this form by introducing a new variable.
 See, e.g., \cite{Bolty} for a detailed description.
\end{remark}

 Throughout this paper, we make the following assumptions.
\begin{itemize}
\item[(A1)] The functions $f:R^{n_x}\times U\times R\rightarrow
R^{n_x}$ and
$h:R^{n_x}\rightarrow R$ are at least continuously differentiable with respect to all their arguments.\\
\item[(A2)]  For each fixed $p\in R$, there exist positive constants $L$ and $C$ such that the following inequality holds
$$\|f(x,u,p)\|\leq L\|x\|+C,\ \ \forall  x\in R^{n_x} \mbox{ and } u\in \mathcal{U}.$$
\end{itemize}
  From the classical differential equation theory (See, for example, Proposition 5.6.5 in \cite{Polark}), we recall that the system (\ref{s1})
admits a unique solution, $x(t;u,p)$, corresponding to each
$u\in\mathcal{U}$ and $p\sim F\in\mathcal{F}(\mu,\sigma^2)$.
Problem (DROCP) can be roughly stated as: Find a  control
$u\in\mathcal{U}$ such that the worst-case expectation from all
feasible distributions is minimized over $\mathcal{U}$. Obviously,
Problem (DROCP) is a min-max optimal control problem.

\section{Distributionally robust optimal control problem with discrete distribution}
In this section, we focus on the case of discrete distributions.
In this case, we will reformulate the distributionally robust
optimal control as an equivalent combined optimal control and
optimal parameter selection problem by using a dual
transformation. We will then develop an algorithm to solve the
resulting problem based on the parametric sensitivity functions
and the control parametrization method.

\subsection{Problem reformulation and optimality conditions}
Let $p^i$ be a possible value of the parameter $p$, and let $q_i$
be the corresponding probability, i.e., $\mathbb{P}(p=p^i)=q_i$,
$i=1,2,\cdots,m$. We first investigate the inner
$\sup$-optimization problem, in which the value of  $u$ is fixed.
In this context, there are $m$ possible system trajectories due to
$m$ different values of the parameter $p$. Let $x^i(t;u,p^i)$ be
the trajectory of system (\ref{s1}) with $p=p^i$. When there is no
confusion, $x^i(t;u,p^i)$ is written as $x^i$. Each possible
trajectory yields a corresponding  system cost
$h(x^i(t_f;u,p^i))$. The inner subproblem is to evaluate the
worst-case expectation  from all possible distributions, which is
given as follows:
 \begin{eqnarray*}
(\mbox{ISP})&&\sup_{q_i} \ \ \ \ \  \sum_{i=1}^m q_i h(x^i(t_f;u,p^i))\\
&&\mbox{s.t.} \ \ \ \ \ \sum_{i=1}^m q_i=1,\\
&&\ \ \ \ \ \ \  \ \ \sum_{i=1}^m q_ip^i=\mu,\\
&&\ \ \ \ \ \ \ \  \ \sum_{i=1}^m q_i{(p^i)}^2=\mu^2+\sigma^2,\\
&&\ \ \ \ \ \ \   \ \ q_i\geq0, \ \ \ i=1,...,m.
\end{eqnarray*}
 Note that the only variables to be optimized in the above inner
 subproblem are $q_i$, $i=1,2,\cdots,m$. Hence, the
 constraints (\ref{DR-SC}) and (\ref{DR-UC}) are not present in ISP. In addition,
 Problem (ISP) is a linear programming, and its dual is  given as follows:
\begin{eqnarray*}
(\mbox{Dual-ISP})&&\inf_{y} \ \ \ \ \ \ y^{\top}b\\
&&\mbox{s.t.} \  \ \  \  y^{\top}a^i\geq
h(x^i(t_f;u,p^i)),~~i=1,2,\cdots,m.
\end{eqnarray*}
where
\begin{align}
&b:=[1,\mu,\mu^2+\sigma^2]^{\top},
\qquad y:=[y_{1},y_2,y_3]^{\top},\nonumber\\
&a^i:=[1,p^i,{(p^i)}^2]^{\top}, \qquad i=1,2,\cdots,m.\nonumber
\end{align}

There is no duality gap between the inner subproblem and its dual
problem, since the feasible set of Problem (ISP) is nonempty and bounded.
Thus, the original Problem (DROCP) is equivalent to the following problem:
\begin{eqnarray}
(\mbox{Dual-DROCP}):&&\inf_{u,y}\,\,\, y^{\top}b \label{D1}\\
 && {\rm s.t.}\,\,\,y^{\top}a^i\geq h(x^i(t_f;u,p^i)),\ \ i=1,2,\cdots,m, \label{D2}\\
 &&\ \ \ \ \ \ \dot{x}^i(t)=f(x^i,u,p^i), \ \ \ \ t\in [0,t_f],\ i=1,2,\cdots,m,\ \ x^i(0)=x^0,\label{D3}\\
&&\ \ \ \ \ \   u\in\mathcal{U}.
\end{eqnarray}

\begin{remark} Problem (Dual-DROCP) can be regarded as a  combined
optimal control and optimal parameter selection
problem, where $u$ is the control function and $y$ is a
parameter vector to be optimized.
\end{remark}

Let $h_i:=h(x^i(t_f;u,p^i))$ and $f^i:=f(x^i,u,p^i)$.
Combining system (\ref{D3}) and the scalar inequality constraints
(\ref{D2})
to the cost function $y^{\top}b$ with multiplier functions
$\lambda(t):=[\lambda_{i,j}(t)]_{m\times n_{x}}$ and multiplier
vector $\theta:=[\theta_1,\theta_2,\cdots,\theta_m]^{\top}$ yields
the Lagrangian of Problem (Dual-DROCP) as given below.
\begin{align}
\mathcal{L}(u,y)&=y^{\top}b+\sum\limits_{i=1}^m\theta_i(y^{\top}a^i-h(x^i(t_f;u,p^i)))+\sum\limits_{i=1}^{m}\int_{0}^{t_f}\lambda^i(t)[\dot{x}^i-f^i]dt\nonumber\\
&=y^{\top}b+\sum\limits_{i=1}^m\theta_i[y^{\top}a^i-h(x^i(t_f;u,p^i))]-\sum\limits_{i}^{m}\int_{0}^{t_f}\lambda^i(t)f^idt\nonumber\\
&~+\sum\limits_{i}\lambda^i(t_f)x^i(t_f)-\sum\limits_{i}^m\lambda^i(0)x^i(0)-\sum\limits_{i}^m\int_{0}^{t_f}\dot{\lambda}^i(t)x^i(t)dt,\nonumber
\end{align}
where $\lambda^i(t):=[\lambda_{i,1}(t),\lambda_{i,2}(t),\cdots,\lambda_{i,n_{x}}(t)]$. Let $\tilde{y}=y+\epsilon\delta y$ and $\tilde{u}=u+\epsilon\delta u$. Then
\begin{eqnarray*}
\triangle\mathcal{L}&=&\mathcal{L}(u+\epsilon \delta u,y+\epsilon \delta y)-\mathcal{L}(u,y)\\
&=&\epsilon\delta y^{\top}(b+\sum\limits_{i=1}^m\theta_i a^i)-\epsilon\sum\limits_{i=1}^m\theta_i\frac{\partial h_i}{\partial x^i}\delta x^i(t_f)-\epsilon\int_{0}^{t_f}\Big[\sum\limits_{i=1}^m\lambda^i(t)\Big(\frac{\partial f^i}{\partial x^i}\delta x^i(t)+\frac{\partial f^i}{\partial u}\delta u\Big)\Big]dt\\
&+&\epsilon\sum\limits_{i=1}^m\lambda^i(t_f)\delta x^i(t_f)-\epsilon\int_{0}^{t_f}\sum\limits_{i=1}^m\dot{\lambda}^i(t)\delta x^i(t)dt+ o(\epsilon).
\end{eqnarray*}

Based on the fundamental variational principle \cite{APPL}, we have
the necessary optimality conditions of Problem (Dual-DROCP) given
in the following as a theorem.

\begin{theorem}\label{opti}
Consider Problem (Dual-DROCP). If $u^*(t)\in U$ is an optimal
control, and $x^*(t)$ is the corresponding state.  Then there
exist costate functions
$\lambda^i(t)=[\lambda_{i,1}(t),\lambda_{i,2}(t),\cdots,\lambda_{i,n_{x}}(t)]$,
$i=1,2,\cdots,m$, and a multiplier vector
$\theta=[\theta_1,\theta_2,\cdots,\theta_m]^{\top}$ with
$\theta_i\geq 0$, $i=1,2,\cdots,m$, such that
\begin{itemize}
\item[$(a)$] $b+\sum\limits_{i=1}^m\theta_i a^i=0$;

\item[$(b)$] $\dot{\lambda}^i(t)=-\lambda^i(t)\displaystyle\frac{\partial f^i}{\partial x^i}$ and the terminal condition $\lambda^i(t_f)=\theta_i\displaystyle\frac{\partial h_i}{\partial x^i}$, $\ i=1,2\cdots,m$;

\item[$(c)$] $\sum\limits_{i=1}^m\lambda^i(t)\displaystyle\frac{\partial f^i}{\partial u}=0$;

\item[$(d)$] $\theta_i\cdot(y^{\top}a^i-h(x^i(t_f;u,p^i)))=0$,
$i=1,2,\cdots,m$.
\end{itemize}
\end{theorem}

\subsection{The optimization algorithm}

Assume that $(u^*(t),y^*)$ is a solution of Problem
(Dual-DROCP). Clearly, the optimal control function, $u^*(t)$,
for Problem (Dual-DROCP) is also the optimal solution of Problem (DROCP). Then, the optimal distribution, $q^*$, can be
obtained by solving Problem (ISP) with $u=u^*(t)$. The algorithm framework
for the solution of Problem (DROCP) is presented as follows.

\begin{itemize}
\item Step 1. Solve Problem (Dual-DROCP), denote the solution by
$(u^*(t),y^*)$;
 \item Step 2. For each possible parameter
$p^i$, $i=1,2,\cdots,m$, compute the optimal trajectories,
$x(t;u^*,p^i)$, and the corresponding cost $h(x^i(t_f;u,p^i))$;
 \item Step
3. Compute the optimal solution $q^*$ of Problem (ISP) by
using linear programming solver.
\end{itemize}

\begin{remark}
Note that we can obtain the most robust optimal control
$u^*(t)$  by only solving Problem (Dual-DROCP). The corresponding ``worst''
distribution $q^*$ shall also be obtained for many practical problems. In this case, we can
estimate the distribution of the performance under the most robust
optimal control and the corresponding distribution of the uncertain parameter. Therefore,
Problem (DROCP) is solved completely by further carrying out Step 2 and Step 3
in the above algorithm framework.
\end{remark}

For the above algorithm framework, Step 2 and Step 3 can be
computed readily. Thus, the remaining problem is on how to solve Problem (Dual-DROCP).

\subsubsection{Control parametrization}

Let $0=t_0<t_1<t_2<\cdots<t_n=t_f$ be the partition grids of the time horizon $[0,t_f]$. On the control
parametrization framework, the control function $u(t)$ is approximated by a piecewise constant function or a piecewise linear function, where the heights of these approximate functions are decision variables. In fact, the control function can be approximated by a linear combination of any appropriate set of basis
functions. Thus, Problem (Dual-DROCP) is approximated as a finite-dimensional
optimization problem, where the coefficients of the basis functions are regarded as decision variables. In this paper, the control is approximated as a piecewise constant function in the form as given below:

\begin{eqnarray}\label{uv}
u^\iota(t)=\sum\limits_{k=1}^n v^{k}\chi_{I_{k}}(t), \ \ t\in[0,t_f],
\end{eqnarray}
where $v^k=[v^k_1,v^k_2,\cdots,v^k_{n_u}]^{\top}\in U$,
$I_k=[t_{k-1},t_k)$, $k=1,2,\cdots,n$, and $\chi_I$ denotes the
characteristic function of $I$. Define
$v=[({v^1})^{\top},({v^2})^{\top},\cdots,({v^n})^{\top}]^{\top}$ and
$\mathcal{V}=\prod\limits_{k=1}^n U$. Clearly, the control $u$
defined in the form of (\ref{uv}) is one to one corresponding with the $n\times
n_{u}$ control parameter vector $v$.
Let $x(t;v,p^i)$ be the solution of system (\ref{s1})
 corresponding to $(v,p^i)$.  With some abuse of notation, $x(t;v,p^i)$ is abbreviated as $x^{v,i}(t)$ or $x^{v,i}$ when no confusion can arise.
Then, the parameterized problem for Problem (Dual-DROCP) can be stated as given below:
\begin{eqnarray}
(\mbox{Discre-Dual-DROCP}):&&\inf_{v,y}\,\,\, y^{\top}b \label{DD1}\\
 && {\rm s.t.}\,\,\,y^{\top}a^i\geq h(x^i(t_f;u,p^i)), \label{DD2-new}\\
 &&\ \ \ \ \ \ \dot{x}^i(t)=f(x^i,v,p^i),   \ t\in [0,t_f],\ i=1,2,\cdots,m, \ x^i(0)=x^0, \label{DD3}\\
&& \ \ \ \ \ \ v\in\mathcal{V}.
\end{eqnarray}

\subsubsection{Gradient formulas}

Problem (Discre-Dual-DROCP) is essentially a finite-dimensional
optimization problem, which can be solved readily by various
optimization techniques. In general, the values of the objective
function and the constraint functions and their respective
gradients are required to be computed at each iteration of the
optimization procedure. The gradient of the objective function is
obvious since it is only a linear function of $y$. The gradients
of the constraint functions can be evaluated by solving either the
adjoint equations  (see, for example, \cite{Teo1989}) or
the sensitivity function  (see, for example,
\cite{Ryan2012,Feehery1998,Rose2000}). In this paper, the method based on
the sensitivity function is used. The parametric sensitivity
system and the gradient formulas are given in the following as a
theorem.
\begin{theorem}\label{para}
Consider system (\ref{DD3}). Let $x(t;v,p^i)$ be the solution
and let $n_v=n_u\times n$ be the dimension of $v$. Let
$s^j(t;v,p^i)=[s_1^j(t;v,p^i),s_2^j(t;v,p^i),\cdots,s_{n_x}^j(t;v,p^i)]^{\top}$
be the parametric sensitivity function of system (\ref{DD3})
with respect to $v_j$,  i.e.,
\begin{eqnarray}\label{psf}
s^j(t;v,p^i)=\frac{\partial x(t;v,p^i)}{\partial v_j}, \ \ j=1,2,\cdots,n_v.
\end{eqnarray}
Then, $s^j(t;v,p^i)$ is the unique solution of  the following differential equation system
\begin{align}\label{ps}
\left\{
\begin{array}{ccc}
\displaystyle\frac{d s^j}{dt}&&=\displaystyle\frac{\partial f}{\partial x}(x^i,v,p^i)s^j+\displaystyle\frac{\partial f}{\partial u}(x^i,v,p^i)E_l\chi_{I_k}(t),\ \ t\in[0,t_f],\\
s^j(0)&&=0 .~~~~~~~~~~~~~~~~~~~~~~~~~~~~~~~~~~~~~~~~~~~~~~~~~~~~~~~~~~~
\end{array}\right.
\end{align}
where $I_k:=[t_{k-1},t_k)$, $k=[\frac{j}{n_u}]+1$, $l=j \mbox {
mod } n_u$, $E_l$ is an $n_u$-dimensional column vector whose
$l$-th component  is one and all other components are zeros.

Furthermore, the gradients of the constraint functions
$h(x^{v,i}(t_f))$,  $i=1,2,\cdots,m$, with respect to $v$,  are
given by
\begin{eqnarray}
\nabla_v h(x^{v,i}(t_f))=\frac{\partial h}{\partial x}S(t_f),
\end{eqnarray}
where $S=[(s^1)^{\top},(s^2)^{\top},\cdots,(s^{n_v})^{\top}]^{\top}$,
and its components are given by (\ref{ps}).
\end{theorem}

\subsubsection{Algorithm procedure}

Problem  (Discre-Dual-DROCP) differs from the standard
mathematical programming problems in the sense that it involves
the dynamic system (\ref{DD3}) and the end-point constraints
(\ref{DD2-new}). The dynamic constraint (\ref{DD3}) as well as the
systems of differential equations of the parametric sensitivity
functions are solved by an ordinary differential equation (ODE)
solver in each iteration of the optimization procedure. The
end-point constraints (\ref{DD2-new}) are handled as follows.
Define
\begin{eqnarray}
g_i(x,v,y):=h(x^{v,i}(t_f))-y^{\top} a^i
\end{eqnarray}
and $G_0({x},v,y):=\max\limits_{i}\{0,{g}_i({x} ,v,y)\}$.
Constraints (\ref{DD2-new}) are equivalent to the following
equality constraint:
\begin{eqnarray}
{G}_0({x},v,y)=0.\label{Gc}
\end{eqnarray}
However, $G_0 ({x},v,y)$ is nonsmooth in $(v,y)$. Standard
optimization routines would have difficulties in handling this
type of equality constraints. A widely used  smoothing technique
\cite{Ryan2009} is to approximate ${g} _i$ by
\begin{align}
{g} _i^{\epsilon}({x} ,v,y):=\left\{
\begin{array}{lll}
0,&\mbox{if}~~{g}_i({x} ,v,y)<-\epsilon,\\
\displaystyle\frac{({g} _i({x},v,y)+\epsilon)^2}{4\epsilon}, &\mbox{if}~~-\epsilon\leq {g}_i({x},v,y)\leq \epsilon,\\
{g}_i({x},v,y), &\mbox{if}~~{g}_i({x},v,y)>\epsilon.
\end{array}
\right.
\end{align}

By using the quadratic penalty function, Problem
(Discre-Dual-DROCP) is finally approximated by
\begin{eqnarray}
(\mbox{QP-Dual-DROCP}):&&\inf_{v,y}\,\,\,
\mathcal{J}({x},v,y):=y^{\top}b
+\frac{\varrho}{2}(G_\epsilon({x},v,y))^2 \label{LD1}\\
 &&\ \ \ \ \ \ \dot{x}^i(t)=f(x^i,v,p^i), \ \ \ \ t\in [0,t_f],\ i=1,2,\cdots,m,\label{LD3}\\
 &&\ \ \ \ \ \ x^i(0)=x^0.  \label{LD4}
\end{eqnarray}
where
$G_\epsilon({x},v,y):=\sum\limits_{i=1}^{m}{g}_i^\epsilon({x},v,y)$
and $\varrho$ is  the penalty parameter. The algorithm framework
for the solution of Problem (QP-Dual-DROCP), which is constructed
based on Algorithm 17.4 in \cite{Noceal2006}, is stated as
follows.

\begin{center}
\bf{Algorithm 3.1}
\end{center}
\begin{itemize}
\item Initialize:  \\
Choose an initial point $(v^0,y^0)$. Choose convergence tolerances
$\eta_*$ and $\omega_*$. Choose positive constants
$\bar{\varrho}$, $\alpha_1>1$, $\alpha_2<1$ and $\alpha_3<1$.
Set $\varrho_0=\bar{\varrho}$, $\omega_0=1/\varrho_0, \eta_0=1/\varrho_0^{0.1}, k=0$;  \\
\item Repeat\\
\begin{itemize}
\item[(S1)]For $i=1,2,\cdots, m$, integrate system
(\ref{LD3})-(\ref{LD4}) together with the parametric sensitivity
systems (\ref{ps}) forward in time from 0 to $t_f$. \item[(S2)]
Evaluate the value of the merit function $\mathcal{J}$ and its
gradients, denoted by $\mathcal{J}({x},v^k,y^k)$ and
$\nabla\mathcal{J}({x},v^k,y^k)$, respectively. \item[(S3)] If
$\|P_\mathcal{V}[\nabla\mathcal{J}({x},v^k,y^k)]\|\leq \omega_k$,
where $P_\mathcal{V}d$ is the partial projection of the vector
$d\in R^{n_v+3}$ onto the rectangular box $\mathcal{V}=[v_*,v^*]$
at the current point $(v^k,y^k)$, defined by
\begin{eqnarray}
P_\mathcal{V}d=\left\{\begin{array}{lll}
\min \{0,d_i\},\ \ \mbox{ if } i\leq n_v \mbox{ and } v_i=v_{i*},\\
d_i,\ ~\quad \ \qquad \ \mbox{ if }   i\leq n_v \mbox{ and }v_i\in(v_{i*},v^{i*}),\mbox{ or } i>n_v, \ \ \mbox{ for all } i=1,2,\cdots,n_{v+3},\\
\max\{0,d_i\},\ \ \mbox{ if } i\leq n_v \mbox{ and } v_i=v^{i*}.
\end{array}
\right.
\end{eqnarray}
Then goto (S4-1). Otherwise, goto (S4-2). \item[(S4-1)] If
${G}_\epsilon({x}^k,v^k,y^k)\leq \eta_k$, goto (S5-1); otherwise,
goto (S5-3)

\item[(S4-2)] Using a line search method to find the next point
$(v^{k+1},y^{k+1})$, replace $(v^k,y^k)$ by  $(v^{k+1},y^{k+1})$,
and goto (S1). \item[(S5-1)]---Stopping criterion

If ${G}_\epsilon({x}^k,v^k,y^k)\leq \eta_*$ and
$\|P_\mathcal{V}[\nabla\mathcal{J}({x},v^k,y^k)]\|\leq \omega_*$,
stop. Record the approximate solution $(v^{k},y^k)$ obtained. Otherwise, goto
(S5-3). \item[(S5-2)] ---Tighten tolerance

Set $\eta_{k+1}:=\alpha_3\eta_k$ and goto (S5-3).
\item[(S5-3)]---Increase penalty parameter

 Set  $\varrho_{k+1}:=\alpha_1 \varrho_{k}$,
$\omega_{k+1}:=\alpha_2\omega_{k+1}$, $k:=k+1$, and goto (S1).
\end{itemize}

\end{itemize}
\begin{remark}
Note that since ${x} $ is an intermediate variable depending on
$v$ and $p_i$ rather than an independent variable, the merit
function $\mathcal{J}$ can, in essence, be regarded as a function
of $v$ and $y$. Thus, the gradient of $\mathcal{J}$ only composes
of the partial derivatives of $\mathcal{J}$ with respect to $v$
and $y$; it is a vector of dimension $n_v+3$.
\end{remark}

\section{The case of continuous distributions}

For the case of continuous distributions, the cost function
$h(x(t_f;u,p))$ can be considered as a function of $u$ and $p$,
because the state $x$ is only an intermediate variable depending on $u$ and $p$. For a fixed
$u$, the inner sub-problem is given as follows.
 \begin{eqnarray}
(\mbox{CISP})&&\sup_{F} \ \ \ \ \  \int_{\mathcal{F}}h(x(t_f;u,p))dF(p)\nonumber\\
&&\mbox{s.t.} \ \ \ \ \ \int_{\mathcal{F}}dF(p)=1,\label{sup-p1}\\
&&\ \ \ \ \ \ \  \ \ \int_{\mathcal{F}}pdF(p)=\mu,\label{sup-p2}\\
&&\ \ \ \ \ \ \ \  \ \int_{\mathcal{F}}p^2dF(p)=\mu^2+\sigma^2,\label{sup-p3}\\
&&\ \ \ \ \ \ \   \ \ dF(p)\geq0.\label{sup-p4}
\end{eqnarray}

To extent the results obtained for the case of discrete
distributions detailed in the previous section to the case of
continuous distributions, we propose a scheme for the
discretization of the continuous stochastic variable
 based on the control parametrization method.

Suppose that the uncertain parameter $p$ is disturbed in an
interval $[p_l,p_u]$.  Let $\psi:[p_l,p_u]\rightarrow[0,\infty)$
be an element of $\mathcal{F}(\mu,\sigma)$, i.e., $\psi$ is a
potential probability density function of $p$ satisfying
(\ref{sup-p2}) and (\ref{sup-p3}).  Let
$p_l=p_0<p_1<p_2<\cdots<p_m=p_u$ be a set of time points on the interval $[p_l,p_u]$.
Denote $[p_{i-1},p_i)$ by $I_i^p$,
$i=1,2,\cdots,m-1$, and $[p_{m-1},p_m]$ by $I_m^p$. Let $\Delta
p_i:=p_i-p_{i-1}$,  and let
\begin{eqnarray}\label{dp}
\Delta p:=\max\limits_{i} \Delta p_i.
\end{eqnarray}
Let $p_d^i$ be an arbitrarily but fixed element chosen from $[p_{i-1},p_i)$.
It is referred to as a characteristic element of this
subinterval. When the uncertain parameter takes values in
$[p_{i-1},p_i)$, it is approximated as $p_d^i$ in the
system. As a result, the uncertain parameter interval $[p_l,p_u]$ is
approximated as a finite set $\{p_d^i\}_{i=1}^m$. Moreover, the
probability $\mathbb{P}(p=p_d^i)$ is defined as
\begin{eqnarray}\label{pmf}
\mathbb{P}(p=p_d^i)=\int_{p_{i-1}}^{p_i}\psi(p)dp:=q_d^i, \ \ i=1,2,\cdots,m.
\end{eqnarray}

On the discretization of the continuous distribution, the cost function is approximated as follows:
\begin{eqnarray}\label{gd}
\int_{\mathcal{F}}h(x(t_f;u,p))dF(p)\thickapprox\sum\limits_{i=1}^m q_d^ih(x^i(t_f;u,p_d^i)).
\end{eqnarray}
  The same idea can be used for the constraints. Thus, we can  approximate the inner
  sub-problem (CISP) by the following discrete-distribution problem
\begin{eqnarray}
(\mbox{DISP})&&\sup_{F} \ \ \ \ \  \sum\limits_{i=1}^m q_d^ih(x^i(t_f;u,p_d^i))\nonumber\\
&&\mbox{s.t.} \ \ \ \  \sum\limits_{i=1}^m q_d^i=1,\label{sup-q1}\\
&&\ \ \ \ \ \ \  \ \ \sum\limits_{i=1}^m p_d^i q_d^i=\mu,\label{sup-q2}\\
&&\ \ \ \ \ \ \ \  \ \sum\limits_{i=1}^m (p_d^{i})^2 q_d^i=\mu^2+\sigma^2,\label{sup-q3}\\
&&\ \ \ \ \ \ \   \ \ q_d^i\geq0.\label{sup-q4}
\end{eqnarray}

Note that Problem (DISP) is the same as Problem (ISP) detailed in
the previous section. That is, it is a distributionally robust
optimal control problem with discrete distribution. If the above
discretization method is convergent, the solution of Problem
(DROCP) with continuous distribution can be approximately obtained
through solving a sequence of problems with discrete
distributions.
 Therefore, we only need to verify the convergence of the
 above discretization scheme, which  will be proved by
 investigating the relationships of the cost function and
 constraint functions between Problem (CISP) and Problem (DISP).

From (\ref{pmf}), it is obvious that constraints (\ref{sup-p1})
and (\ref{sup-q1}) are consistent. Besides,  inequality
(\ref{sup-p4}) in Problem (CISP) also implies  inequality
(\ref{sup-q4}) in Problem (DISP). Therefore, we only need to
evaluate the differences of the cost functions and the two
constraints between Problem (CISP) and Problem (DISP). Details are
given in the following two theorems.
\begin{theorem} \label{th-x}
Given $u\in\mathcal{U}$ and any $p\in I_p^i$, let $x(\cdot;p)$ and $x^i(\cdot;p_d^i)$ be,
respectively, the solution of
\begin{eqnarray*}
\left\{
\begin{array}{lll}
\dot{x}=f(x,u,p),\\
x(0)=x^0,
\end{array}
\right.\ \ t\in[0,t_f]
\end{eqnarray*}
and the solution of
\begin{eqnarray}
\left\{
\begin{array}{lll}
\dot{x}=f(x,u,p_d^i),\\
x(0)=x^0,
\end{array}
\right.\ \ t\in[0,t_f].
\end{eqnarray}
Then, there exists a constant $L_1>0$, which is independent of $p$ and $p_d^i$, such that the inequality
\begin{eqnarray*}
\|x(t;p)-x^i(t;p_d^i)\|\leq L_1\Delta p,
\end{eqnarray*}
holds for all $t\in[0,t_f]$ with $\Delta p$ defined in (\ref{dp}).
\end{theorem}
\begin{proof}
Given $u\in\mathcal{U}$ and any $p\in I_p^i$, the solution of
system (\ref{s1}) can be expressed as
$$x(t;p)=\int_0^{t}f(x,u,p)dt, \ \ \ \forall t\in[0,t_f].$$
It follows that
\begin{eqnarray}
\|x(t;p)-x^i(t;p_d^i)\|&=&\|\int_0^t f(x,u,p)dt-\int_0^t f(x,u,p_d^i)dt\|\nonumber
\\
&\leq&\int_0^t\| f(x,u,p)- f(x,u,p_d^i)\|dt.
\end{eqnarray}
Since $f$ is at least continuously differentiable with respect to
$p$, it also satisfies Lipschitz condition in $p$ on $[p_l,p_u]$, that is, there
exists a constant $L_1$ such that
\begin{eqnarray}
\| f(x,u,p)- f(x,u,p')\|\leq L_1|p-p'|, \  \ \forall p,p'\in[p_l,p_u].
\end{eqnarray}
Therefore, we have
\begin{eqnarray}
\|x(t;p)-x^i(t;p_d^i)\|\leq L_1|p-p_d^i|\leq L_1 \Delta p, \ \ \forall t\in[0,t_f]\nonumber
\end{eqnarray}
where $\Delta p$ is defined in (\ref{dp}). This complete the proof.
\end{proof}

\begin{theorem}\label{th-conv}
Let $\psi(p)$ be a probability density function satisfying
(\ref{sup-p2}) and (\ref{sup-p3}) and let $F$ be the corresponding
distribution function. Then, for any control $u\in\mathcal{U}$ and
$\epsilon>0$, there exists $\delta>0$  such that the following
inequalities
\begin{itemize}
\item[(a)]
$\Big|\displaystyle\int_\mathcal{F}pdF(p)-\displaystyle\sum\limits_{i=1}^m
p_d^i\cdot q_d^i\Big|=\Big|\displaystyle\int_{p_l}^{p_u}
p\psi(p)dp-\displaystyle\sum\limits_{i=1}^m p_d^i\cdot
q_d^i\Big|\leq\epsilon;$ \item[(b)]
$\Big|\displaystyle\int_\mathcal{F}p^2dF(p)-\displaystyle\sum\limits_{i=1}^m(p_d^i)^2\cdot
q_d^i\Big|=\Big|\displaystyle\int_{p_l}^{p_u}
p^2\psi(p)dp-\displaystyle\sum\limits_{i=1}^m (p_d^i)^2\cdot
q_d^i\Big|\leq\epsilon;$ \item[(c)]
$\Big|\displaystyle\int_\mathcal{F}
h(x(t_f;u,p))dF(p)-\displaystyle\sum\limits_{i=1}^m q_d^i
h(x^i(t_f;u,p_d^i)) \Big|\leq \epsilon.$
\end{itemize}
hold provided the  grid size $\Delta p\leq\delta$.
\end{theorem}

\begin{proof}
(a) The equality holds directly from the definition.
Thus, it remains to prove the validity of the inequality. From (\ref{pmf}), we have
\begin{eqnarray*}
\Big|\displaystyle\int_{p_l}^{p_u} p\psi(p)dp-\displaystyle\sum\limits_{i=1}^m p_d^i\cdot q_d^i\Big|&=&
\Big|\displaystyle\int_{p_l}^{p_u} p\psi(p)dp-\displaystyle\sum\limits_{i=1}^m \ p_d^i\int_{p_{i-1}}^{p_i}\psi(p)dp\Big|\\
&=&\Big|\displaystyle\int_{p_l}^{p_u} p\psi(p)dp-\displaystyle\sum\limits_{i=1}^m  \int_{p_{i-1}}^{p_i}p_d^i\psi(p)dp\Big|\\
&=&\Big|\displaystyle\sum\limits_{i=1}^m\int_{p_{i-1}}^{p_i} p\psi(p)dp-\displaystyle\sum\limits_{i=1}^m  \int_{p_{i-1}}^{p_i}p_d^i\psi(p)dp\Big|\\
&\leq&\sum\limits_{i=1}^m\int_{p_{i-1}}^{p_i}\Big|p-p_d^i\Big|\psi(p)dp\leq \sum\limits_{i=1}^m \Delta p\int_{p_{i-1}}^{p_i}\psi(p)dp=\Delta p
\end{eqnarray*}

(b) Similar to the derivation given in (a), we obtain
\begin{eqnarray*}
\Big|\displaystyle\int_{p_l}^{p_u} p^2\psi(p)dp-\displaystyle\sum\limits_{i=1}^m (p_d^i)^2\cdot q_d^i\Big|\leq\sum\limits_{i=1}^m\int_{p_{i-1}}^{p_i}\Big|p^2-(p_d^i)^2\Big|\psi(p)dp\leq 2p_u\Delta p\sum\limits_{i=1}^m \int_{p_{i-1}}^{p_i}\psi(p)dp=2p_u\Delta p
\end{eqnarray*}

(c)  Similarly, we have \begin{eqnarray}
&&\Big|\displaystyle\int_{p_l}^{p_u} h(x(t_f;u,p))\psi(p)dp-\displaystyle\sum\limits_{i=1}^m q_d^i h(x^i(t_f;u,p_d^i))\Big|\nonumber\\
=&&
\Big|\displaystyle\sum\limits_{i=1}^m \ \int_{p_{i-1}}^{p_i} \Big[h(x(t_f;u,p))-h(x^i(t_f;u,p_d^i))\Big]\psi(p)dp\Big|\nonumber\\
\leq&& \displaystyle\sum\limits_{i=1}^m \ \int_{p_{i-1}}^{p_i}\Big|h(x(t_f;u,p))-h(x^i(t_f;u,p_d^i))\Big|\psi(p)dp\label{c-1}
\end{eqnarray}
Since $h$ is continuously differentiable in $x$, there exists, for any $\epsilon$, a $\delta_1>0$ such that
\begin{eqnarray}\label{h-ineq}
\Big|h(x)-h(x')\Big|\leq \epsilon, \ \ \ \mbox{if} \ \ \|x-x'\|\leq\delta_1.
\end{eqnarray}
From Theorem \ref{th-x}, it follows that the following inequality
\begin{eqnarray}\label{x-ineq}
\|x(t_f;u,p)-x(t_f;u,p_d^i)\|\leq \delta_1, \forall p\in[p_{i-1},p_i)
\end{eqnarray}
holds if $\Delta p\leq\displaystyle\frac{\delta_1}{L_1}$.
Substitute (\ref{h-ineq}) and (\ref{x-ineq}) into  (\ref{c-1}). If
 \begin{align}
\Delta p\leq \frac{\delta_1}{L_1},\nonumber
\end{align}
then
 \begin{eqnarray}
\Big|\displaystyle\int_{p_l}^{p_u} h(x(t_f;u,p))\psi(p)dp-\displaystyle\sum\limits_{i=1}^m q_d^i h(x^i(t_f;u,p_d^i))\Big|\leq\sum\limits_{i=1}^m\int_{p_{i-1}}^{p_i}\epsilon \psi(p)dp=\epsilon.\nonumber
\end{eqnarray}
Let
$\delta:=\max\{\epsilon,\displaystyle\frac{\epsilon}{2p_u},\frac{\delta_1}{L_1}\}$.
Then, we conclude that  inequalities (a), (b), (c) hold if $\Delta
p\leq\delta$. Hence, the proof is completed.
\end{proof}
\begin{remark}
In Theorem \ref{th-conv}, the relationships of the cost functions
and the constraints between Problem (DROCP)   and its
approximation problem with discrete distributions are given.  Note
that Theorem \ref{th-conv}
 is not related to the issue of local or global
optima. It holds for all controls and all feasible probability
density functions, and hence also holds for global and local
optima.
\end{remark}

\section{Illustration example}
 In this section, we choose an example  to
 illustrate the application of distributionally robust optimal
 control model and to test the performance of the proposed algorithm.
 The illustration example is a distributionally robust optimal control of a microbial
 fed-batch process \cite{Ye2014}, which is stated as follows.

Let $X$ be the concentration of biomass (g/L), $S$ be the
concentration of substrate (g/L) and $V$ be the volume of the
solution (L). The control system of the fed-batch process is
described by
\begin{eqnarray}
&&\dot{X}=(\mu_X-d_X)X,\\
&&\dot{S}=-q_SX+\frac{\rho_S-S}{V}u(t), \\
&&\dot{V}=u(t),
\end{eqnarray}
where $u(t)$ is the input control of the substrate. $d_X$ is the
specific decay rate of cells, and $\rho_S$ is the concentration of
substrate in feed medium. $\mu_X$ is the specific growth rate of
biomass, and $q_S$ is the specific consumption rate of substrate,
which are, respectively, expressed as
\begin{eqnarray}
&&\mu_X=\mu_m\frac{S}{S+K_S}(1-\frac{S}{S^*}),\label{eq-mu}\\
&&q_S=m_S+\frac{\mu_X}{Y_S}.\label{eq-qs}
\end{eqnarray}
In (\ref{eq-mu}), $\mu_m$ is the maximum specific growth rate,
$K_S$ is the saturation constant, and $S^*$ is the critical
concentration of the substrate
 above which cells cease to grow.  In (\ref{eq-mu}), $m_S$ and
 $Y_S$ are, respectively,
  the maintenance requirement of substrate and the maximum growth
  yield. The above system is a typical kinetic model used in microbial
  fermentation process, see, e.g., Ye et al. \cite{Ye2014} and Zeng et al.
  \cite{Zeng1995}. In general, $m_S$ is regarded to be
  constant during the whole fermentation process. However, it is
  well-known that the maintenance consumption of substrate would
  vary during different fermentation stages. Thus, we consider
  $m_S$ as an uncertain parameter in this work.
Assume that the mean and the standard deviation of the uncertain
parameter
 are $m_\mu$ and $m_\sigma$, respectively.

The problem is to control the input $u(t)$ such that the biomass
at the terminal time $t_f$ is maximized. For convenience of
presentation, let $x:=[x_1,x_2,x_3]^\top=[X,S,V]^{\top}$. Define
the admissible set of controls as $\mathcal{U}:=\{u(t)|u_*\leq
u(t)\leq u^*\}$ with $u_*$ and $u^*$ being the minimum and maximum
input rates. Define
\begin{eqnarray}
f(x,u,m_S):=[(\mu_X-d_X)X,-q_SX+\frac{\rho_S-S}{V}u(t),u(t)]^{\top}.\nonumber
\end{eqnarray}
By transforming the time interval [0,$t_f$] into $[0,1]$, the
optimal control problem can be stated as follows:
\begin{eqnarray}
(\mbox{
OCP}):&&\min_{u}\max_F\,\,\, \mathbb{E}_Fh(x(1;u,m_S)):=-x_1(1;u,m_S)\nonumber\\
 && \ \mbox{s.t.} \ \ \dot{x}(t)=t_f f(x,u,m_S), \ \ \ \ t\in [0,1],\ x(0)=x^0,\label{ocp-s1}\\
 &&\ \ \ \ \ \ \ m_S\sim F\in \mathcal{F}(m_\mu,m_\sigma^2)=\{F:\mathbb{E}_{F}(p)=m_\mu,
 \mathbb{E}_{F}(m_S-m_\mu)^2=m_\sigma^2\} ,
\label{ocp-pc}\\
&&\ \ \ \ \ \ \ u\in\mathcal{U}.  \label{ocp2}
\end{eqnarray}

In this numerical example, we set $x^0=[0.1,20,3]^\top$,
$m_\mu=2.2$, $m_\sigma=0.2$, $m_S\in[0.8m_\mu,1.2m_\mu]$,
$d_X=0.05$, $\mu_m=2.7$, $K_S=280$, $Y_S=0.082$, $\rho_S=945$,
$u_*=0$, $u^*=0.04$, $t_f=25$h.

In Algorithm 3.1, even if Problem (Dual-DROCP) is linear with
respect to $y$, we still optimize $y$ together with $u$ by using
the nonlinear optimization techniques. In the numerical
experiments, the time horizon is equidistantly divided into 25
subintervals for the parameterization of the control $u$. Since the
microbial fermentation is a relatively slow time-varying process,
the partition is adequate. In the discretization of the
continuous distribution of the uncertain parameter $m_S$, we
choose ten characteristic elements $\{m_S^i\}_{i=1}^{10}$  over
$[0.8m_\mu,1.2m_\mu]$, where
$m_S^i=0.8m_\mu+\displaystyle\frac{i-1}{9}0.4m_\mu$,
$i=1,2,\cdots,10$.

A good initial guess of the decision variables is important to
help ensure the convergence of the algorithm. We use the following
procedure to generate an initial guess:
randomly generate a control $u$,  and for the fixed $u$, the
variable $y$ is optimized by a linear programming solver and the
performance of (QP-Dual-DROCP) is computed; the process is repeated
$M$(=200) times and the pair $(u,y)$ with the best performance is
set to be an initial guess for a run of Algorithm 3.1. It is worth
mentioning that the generated initial guess, if exists, is a
feasible solution of Problem (Dual-DROCP). Starting from the
initial guess, the proposed algorithm, which is encoded in Matlab
7.0, is implemented on an Intel dual-core i5 with 2450GHz. The
height of the optimal control $u^*$ at each subinterval is listed
in Table 1. Under the optimal input strategy and $m_S=m_S^i,
i=1,2,\dots,10$, the trajectories of the biomass and the substrate
are plotted in Figs. 1 and 2.

After obtaining the optimal solution $(u^*,y^*)$ from Algorithm
3.1, we fix the control $u$ at $u^*$ and optimize $y$ again by
solving Problem (Dual-ISP) with linear programming solver. The
optimal solution is denoted by $y_{u^*}^*$. The values of the cost
function at $(u^*,y^*)$ and $(u^*,y_{u^*}^*)$ are denoted as
$\mathcal{J^*}$ and $\tilde{\mathcal{J}}^*$, respectively. We have
$\mathcal{J^*}=-4.0232$ and $\tilde{\mathcal{J}}^*=-4.1217$. The
difference between $\mathcal{J}^*$ and $\tilde{\mathcal{J}}^*$ is
only 0.0985, which reflects the effectiveness of the proposed
algorithm in some degree. On the other hand, we fix the control
$u$ at $u^*$ and optimize $q$ directly by solving Problem (ISP)
with linear programming solver. The optimal solution of $q$ is
$q^*=[ 0, 0, 0, 0.3223,0.5132, 0, 0, 0, 0,
    0.1645]^\top$.  The terminal concentrations of biomass under the characteristic elements
    $\{m_S^i\}_{i=1}^{10}$ are $[4.1605,
    4.1911,
    4.1998,
    4.1891,
    4.1620,
    4.1210,$ $
    4.0686,
    4.0070,
    3.9382,
    3.8637].$

 \begin{table}[!h]\small
\renewcommand{\arraystretch}{2.0}
\centering\caption{The optimal input strategy control.}
\label{table-ps-rs}
\begin{tabular}{|c|c|c|c|c|c|c|c|c|c|c|c}
\hline  $t$  & [0,1] & [1,2]& [2,3] & [3,4] &
[4,5] & [5,6] & [6,7] & [7,8] & [8,9] & [9,10]\\
\hline
 $u$ &0.0124&
    0.0291&
    0.0276&
    0.0093&
    0.0178&
    0.0137&
    0.0021&
    0.0075&
    0.0048&
    0.0106  \\
\hline $t$  & [10,11] & [11,12]& [12,13] & [13,14] &
[14,15] & [15,16] & [16,17] & [17,18] & [18,19] & [19,20]\\
\hline
 $u$&  0.0042&
    0.0127&
    0.0041&
    0.0195&
    0.0167&
    0.0207&
    0.0203&
    0.0286&
    0.0108&
    0.0344
\\
\hline $t$  & [20,21] & [21,22]& [22,23] & [23,24] &
[24,25] &   &   &   &   &  \\
\hline $u$& 0.0343&
    0.0174&
    0.0383&
    0.0332&
    0.0261 &    &   &    &   &
\\
\hline
\end{tabular}
\end{table}

\begin{figure}[h]
\begin{minipage}[h]{0.48\textwidth}
\centering
  \includegraphics[width=1.1\textwidth]{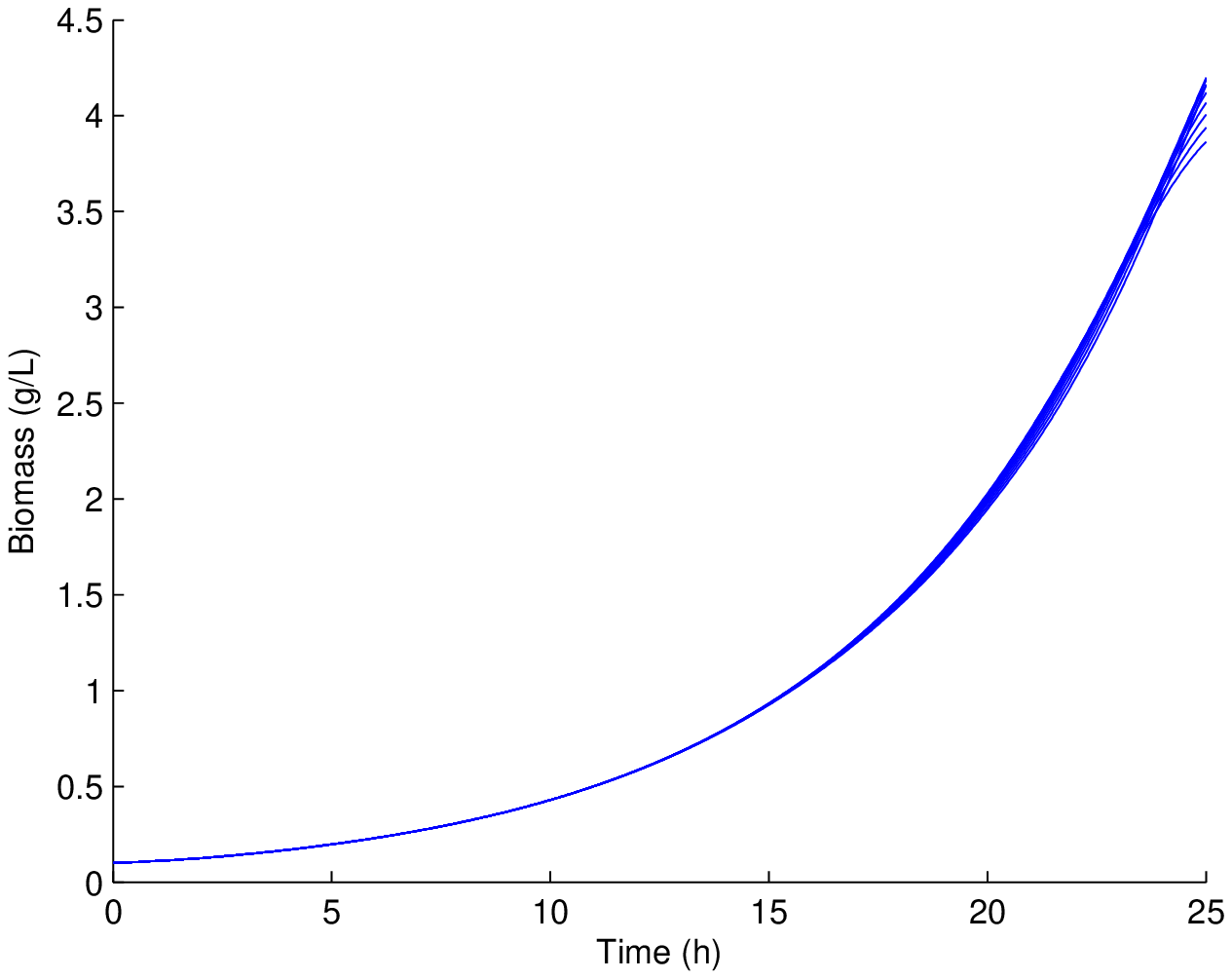}
\caption{The concentrations of biomass under $u=u^*$ and $m_S$
varied from 0.8$m_\mu$ to 1.2$m_\mu$.} \label{figure1}
\end{minipage}
\hfill
\begin{minipage}[h]{0.48\textwidth}
\centering
  \includegraphics[width=1.1\textwidth]{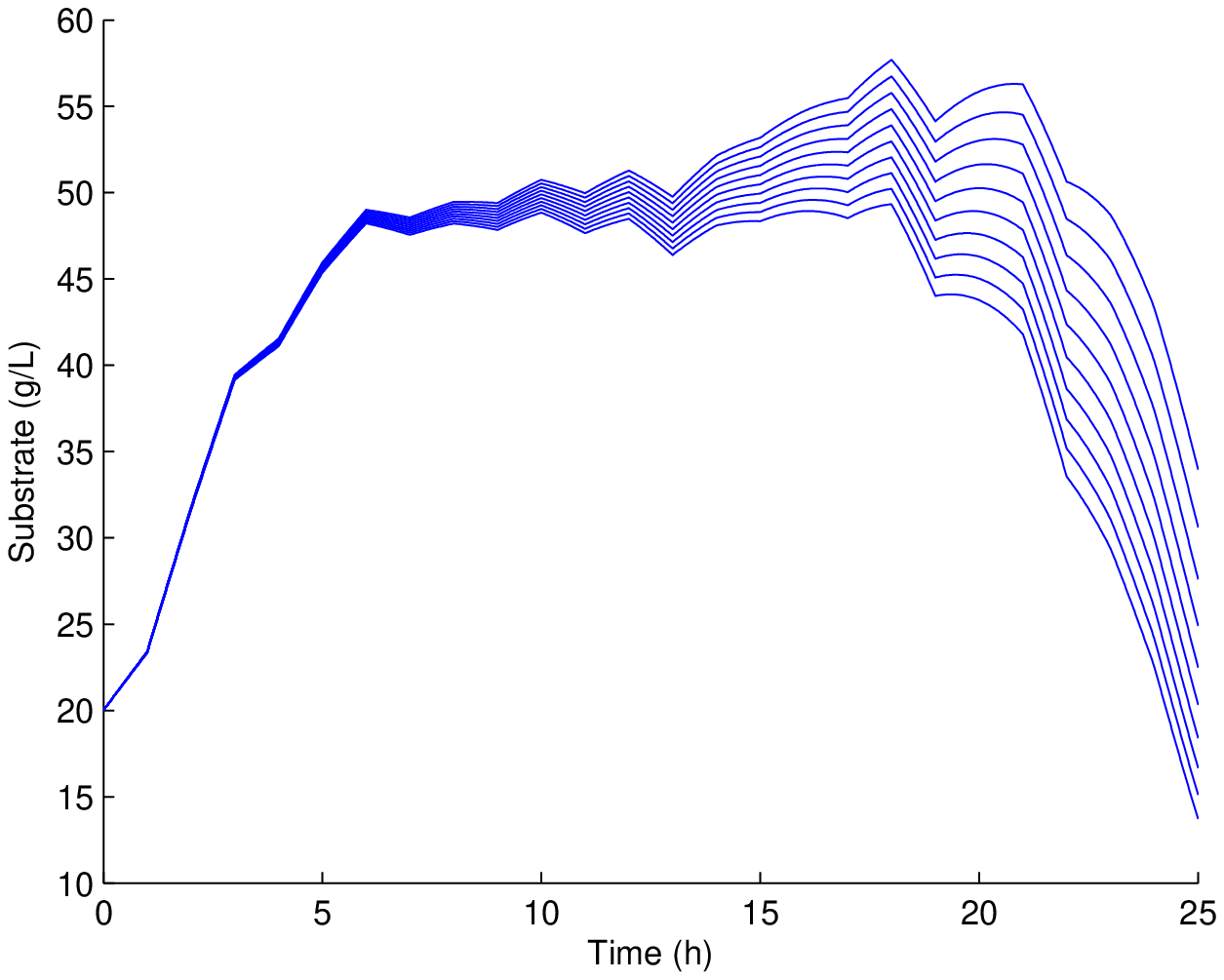}
\caption{The concentrations of substrate under $u=u^*$ and $m_S$
varied from 0.8$m_\mu$ to 1.2$m_\mu$.}\label{figure2}
\end{minipage}
\end{figure}

To illustrate the superiority of the optimal control strategy
obtained from the proposed model, we simulate the system under a
constant input $u(t)\equiv 0.01$. The trajectories of biomass and
substrate with varied $m_S$ under this input strategy are shown in
Figs. 3 and 4.  A comparison of Fig. 1 and Fig. 3 reveals that,
not only the  the terminal concentration of biomass under the
optimal strategy  is significantly higher than that under the
constant control input, but also the variation of biomass
concentration is much smaller than the constant one. This shows that the system under the optimal control strategy could maintain a
good performance even in the ``worst'' case.

\begin{figure}[h]
\begin{minipage}[h]{0.48\textwidth}
\centering
  \includegraphics[width=1.1\textwidth]{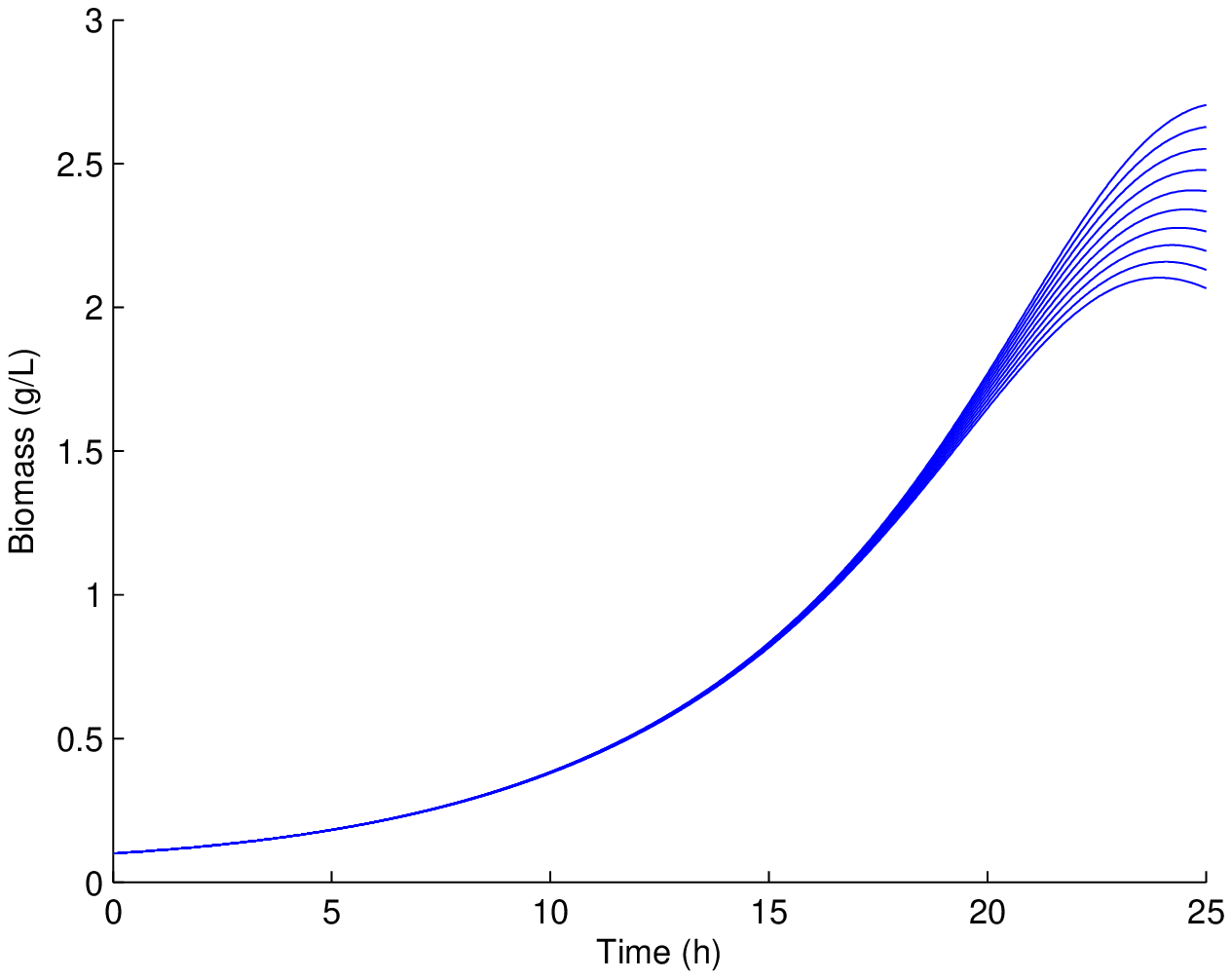}
\caption{The concentrations of biomass under $u=0.01$ and $m_S$
varied from 0.8$m_\mu$ to 1.2$m_\mu$.} \label{figure3}
\end{minipage}
\hfill
\begin{minipage}[h]{0.48\textwidth}
\centering
  \includegraphics[width=1.1\textwidth]{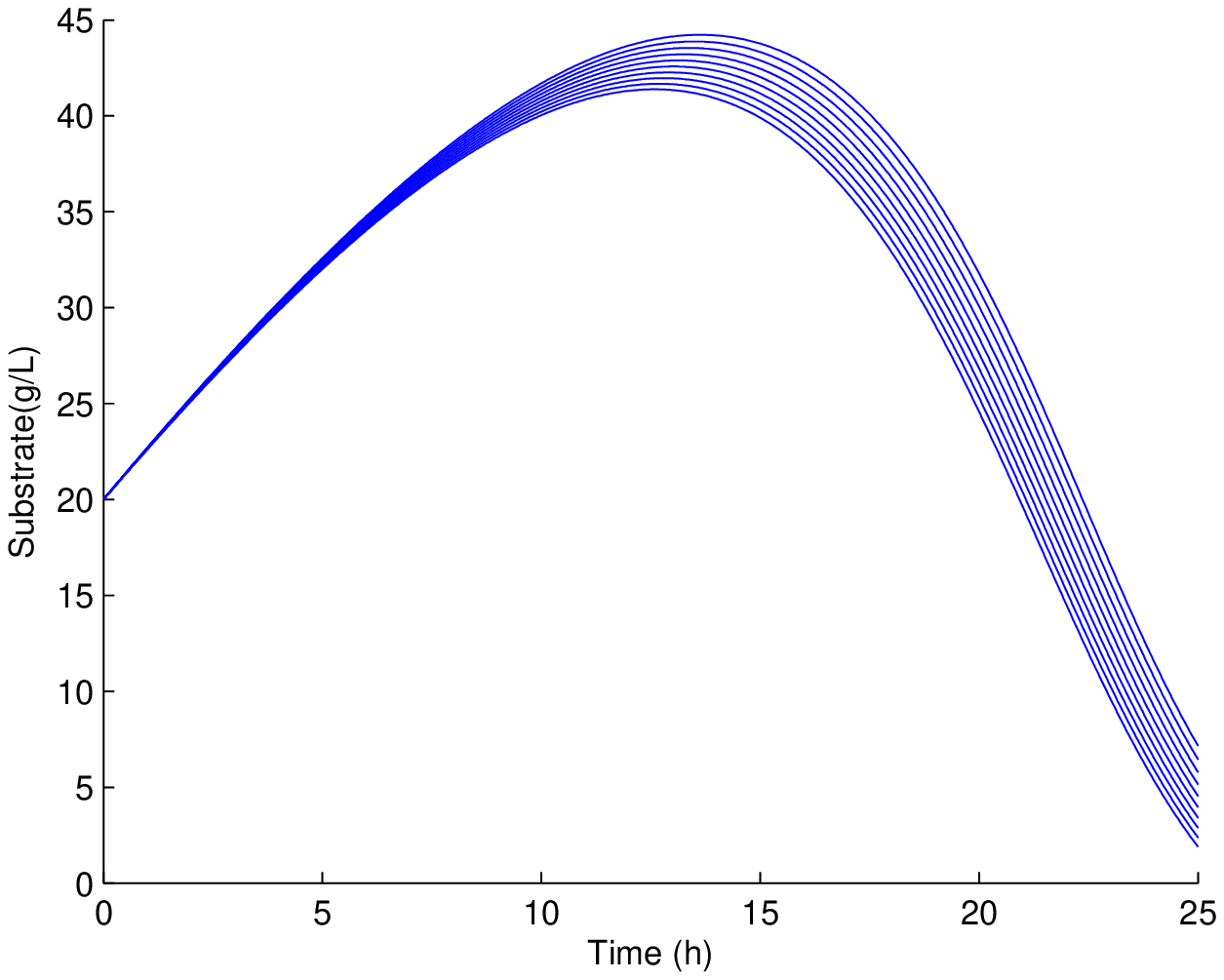}
\caption{The concentrations of substrate under $u=0.01$ and $m_S$
varied from 0.8$m_\mu$ to 1.2$m_\mu$.}\label{figure4}
\end{minipage}
\end{figure}

\section{Conclusion}

This paper introduced an optimal control problem in which both the
objective function and the dynamic constraint contain an uncertain
parameter. Since the distribution of this uncertain parameter is
not exactly known, the objective function is taken as the
worst-case expectation over a set of possible distributions of the
uncertain parameter. To minimize the worst-case expectation over
all possible distributions in an ambiguity set, the stochastic
optimal control problem is converted into a finite-dimensional
optimization problem via duality and discretization. Necessary
conditions of optimality was derived and numerical results for an
illustration example are reported. Numerical results in Section 5
show  the success of the proposed model in producing an optimal
control strategy under which a good performance is achieved. It
also ensures that the variation of the performance is small
subject to the changes in the value of the uncertain parameter.
That is, the system is robust under the optimal control strategy
obtained from the proposed model.

The continuation of this work can be divided into two aspects:
model aspect and algorithm aspect. The  model should take more
factors into account. For example, a further study could be on how
to introduce proper terminal constraints or path constraints into
the model. In the algorithm aspect, the current work transforms
the proposed model into a combined optimal control and optimal
parameter selection problem and solve it by using nonlinear
optimization techniques. However, the special structure of the
problem was not taken into detailed investigation. Problem
(Dual-DROCP) is linear with respect to the optimization vector $y$
but nonlinear with respect the control $u$. An alternative
direction optimization technique could be used to handle these two
kinds of optimization variables separately. For example, the
control $u$ can be fixed first, and the optimal solution $y_u^*$
is easily obtained by solving a linear programming problem. Then,
the control $u$ is regulated by some nonlinear optimization
methods. The procedure is repeated until a satisfactory pair of
$(u,y)$ is found. Some stochastic techniques, such as PSO method,
could also be combined with the alternative direction optimization
technique to regulate the control $u$ in the outer level of the
optimization process.

\section*{Acknowledgements}
This work was supported by the National Natural Science Foundation
for the Youth of China (Grants 11301081, 11401073),
China Postdoctoral Science Foundation (Grant No. 2014M552027), the
Fundamental Research Funds for Central Universities in China
(Grant DUT15LK25), and Provincial National Science Foundation of
Fujian (Grant No. 2014J05001).


\begin{thebibliography}{00}

\bibitem{Soyster} A. L. Soyster.  Convex programming with set-inclusive
constraints and applications to inexact linear programming. Oper.
Res. 1973, 21(5): 1154-1157.

\bibitem{bental} A. Ben-Tal, L. El Ghaoui, A. Nemirovski.
Robust Optimization. Princeton University Press, Princeton, NJ,
2009.

\bibitem{ISP} John R. Birge and F. Louveaux.
Introduction to Stochastic Programming. Springer, 2011.

\bibitem{SP} A. Ruszczynski and A. Shapiro (eds.). Stochastic Programming:
Handbook in Operations Research and
Management Science. Elsevier Science, Amsterdam, 2003.

\bibitem{Goh14} J. Goh,  M.  Sim. Distributionally robust
optimization and its tractable approximations. Oper.  Res. 2010,
58(4-part-1): 902-917.

\bibitem{Sim} M. Sim. Distributionally robust optimization: A marriage of robust
optimization and stochastic programming. 3rd Nordic Optimization
Symposium, March 13-14, 2009, Stockholm, Sweden.

\bibitem{Sim07} X. Chen, M. Sim, P. Sun. A robust optimization perspective on
stochastic programming. Oper. Res. 2007, 55(6): 1058-1071.

\bibitem{Sim09} W. Q. Chen, M. Sim. Goal-driven optimization.
Oper. Res. 2009, 57(2): 342-357.

\bibitem{El} L El Ghaoui, H. Lebret. Robust solutions to least-squares
problems with uncertain data. SIAM J. Matrix Anal. A. 1997, 18(4):
1035-1064.

\bibitem{Ye} D. Erick, and Y. Ye. Distributionally robust optimization under
moment uncertainty with application to data-driven problems. Oper.
Res. 2010, 58(3): 595-612.

\bibitem{Zymler} S Zymler, D Kuhn, B Rustem. Distributionally robust joint
chance constraints with second-order moment information. Math.
Program. 2013, 137(1-2): 167-198.

\bibitem{Mehrotra} S. Mehrotra, H. Zhang. Models and algorithms for
distributionally robust least squares problems. Math. Program.
Ser. A. 2013,  1-19.

\bibitem{Bolty}Vladimir G. Boltyanski and Alexander S. Poznyak. The Robust Maximum Principle-Theory and Applications. Springer Science \& Business Media, 2011.

\bibitem{Polark} E. Polak. Optimization algorithms and consisitent
approximations. Springer-Verlag, New York, Inc., 1997.

\bibitem{APPL} A.E. Bryson, Applied optimal control: optimization, estimation and control. CRC Press, 1975.
\bibitem{Teo1989} K. L. Teo, C. J. Goh. A computational method
for combined optimal parameter selection and optimal control
problems with general constraints. J. Austral. Math. Soc. Ser. B.
1989(30): 350-364.

\bibitem{Ryan2012}R. Loxton, K. L. Teo, V. Rehbock.  Robust Suboptimal Control of Nonlinear Systems.
Appl. Math.  Comput.  2011(217): 6566-6576


\bibitem{Feehery1998} W.F. Feehery, P.I. Barton.
Dynamic optimization with state variable path constraints.
Computer Chem. Enging.  1998(22): 1241-1256.


\bibitem{Rose2000}E. Rosenwasser and  R. Yusupov. Sensitivity of
Automatic Control Systems. Tom Kurfess(eds).  CRC Press, 2000.

\bibitem{Ryan2009}  R.C. Loxton,  K.L. Teo, V. Rehbock, et al, Optimal control problems with a continuous inequality constraint on the state and the control. Automatica, 2009, 45(10): 2250-2257.
\bibitem{Noceal2006} Jorge Nocedal and Stephen J. Wright. Numerical Optimization. Springer, 2006.





\bibitem{Ye2014}J. Ye, H. Xu, E. Feng, Z.
Xiu. Optimization of a fed-batch bioreactor for 1,3-propanediol
production using hybrid nonlinear optimal control. J. Process
Contr. 24(2014): 1556-1569.

\bibitem{Zeng1995}A.P. Zeng, W. D. Deckwer.  A kinetic model for
substrate and energy consumption of microbial growth under
substrate-sufficient conditions. Biotechnol. Prog. 11(1995):
71-79.


\end{thebibliography}
\end{document}